\documentclass[12pt]{amsart}
\usepackage{amscd,amsmath,amsthm,amssymb}

\usepackage{pstcol,pst-plot,pst-3d}%\usepackage[T1]{fontenc}
\usepackage{lmodern,pst-node}%\usepackage{geometric}
\usepackage{multicol}
\usepackage{epic,eepic}
\usepackage{amsfonts,amssymb,amscd,amsmath,enumerate,verbatim}
\psset{unit=0.7cm,linewidth=0.8pt,arrowsize=2.5pt 4}
\newpsstyle{fatline}{linewidth=1.5pt}
\newpsstyle{fyp}{fillstyle=solid,fillcolor=verylight}
\definecolor{verylight}{gray}{0.97}
\definecolor{light}{gray}{0.9}
\definecolor{medium}{gray}{0.85}
\definecolor{dark}{gray}{0.6}
\def\frk{\frak}               % font for "Fraktur"

\def\Phi{{\frk n}}
\def\Phi{{\frk N}}
\def\opn#1#2{\def#1{\operatorname{#2}}} % to make operators
\opn\chara{char} \opn\length{\ell} \opn\pd{pd} \opn\rk{rk}
\opn\projdim{proj\,dim} \opn\injdim{inj\,dim} \opn\rank{rank}
\opn\depth{depth} \opn\grade{grade} \opn\height{height}
\opn\embdim{emb\,dim} \opn\codim{codim}

\opn\Tr{Tr} \opn\bigrank{big\,rank}
\opn\superheight{superheight}\opn\lcm{lcm}
\opn\trdeg{tr\,deg}%\emph{
\opn\reg{reg} \opn\lreg{lreg} \opn\ini{in} \opn\lpd{lpd}
\opn\size{size} \opn\sdepth{sdepth}
\opn\link{link}\opn\fdepth{fdepth}\opn\lex{lex}
\opn\div{div} \opn\Div{Div} \opn\cl{cl} \opn\Cl{Cl}
\opn\Spec{Spec} \opn\Supp{Supp} \opn\supp{supp} \opn\Sing{Sing}
\opn\Ass{Ass} \opn\Min{Min}\opn\Mon{Mon}
\opn\Ann{Ann} \opn\Rad{Rad} \opn\Soc{Soc}
\opn\Im{Im} \opn\Ker{Ker} \opn\Coker{Coker} \opn\Am{Am}
\opn\Hom{Hom} \opn\Tor{Tor} \opn\Ext{Ext} \opn\End{End}
\opn\Aut{Aut} \opn\id{id}

\opn\nat{nat}
\opn\pff{pf}%   \pf exists already
\opn\Pf{Pf} \opn\GL{GL} \opn\SL{SL} \opn\mod{mod} \opn\ord{ord}
\opn\Gin{Gin} \opn\Hilb{Hilb}\opn\sort{sort}
\opn\aff{aff} \opn\con{conv} \opn\relint{relint} \opn\st{st}
\opn\lk{lk} \opn\cn{cn} \opn\core{core} \opn\vol{vol}
\opn\link{link} \opn\star{star}\opn\lex{lex}\opn\set{set}
\opn\gr{gr}

\def\pot#1#2{#1[\kern-0.28ex[#2]\kern-0.28ex]}

\opn\dirlim{\underrightarrow{\lim}}
\opn\inivlim{\underleftarrow{\lim}}
\let\iso=\cong

\def\Implies{\ifmmode\Longrightarrow \else
        \unskip${}\Longrightarrow{}$\ignorespaces\fi}
\def\implies{\ifmmode\Rightarrow \else
        \unskip${}\Rightarrow{}$\ignorespaces\fi}
\def\iff{\ifmmode\Longleftrightarrow \else
        \unskip${}\Longleftrightarrow{}$\ignorespaces\fi}

\let\:=\colon
\newtheorem{Theorem}{Theorem}[section]
\newtheorem{Lemma}[Theorem]{Lemma}
\newtheorem{Corollary}[Theorem]{Corollary}
\newtheorem{Proposition}[Theorem]{Proposition}
\newtheorem{Remark}[Theorem]{Remark}

\let\epsilon\varepsilon
\let\kappa=\varkappa
\def\qed{\ifhmode\textqed\fi
      \ifmmode\ifinner\quad\qedsymbol\else\dispqed\fi\fi}
\def\textqed{\unskip\nobreak\penalty50
       \hskip2em\hbox{}\nobreak\hfil\qedsymbol
       \parfillskip=0pt \finalhyphendemerits=0}
\def\dispqed{\rlap{\qquad\qedsymbol}}

\opn\dis{dis}
\def\pnt{{\raise0.5mm\hbox{\large\bf.}}}

\opn\Lex{Lex}

\begin{document}

\title {Ideals generated by diagonal $2$-minors}

\author {Viviana Ene and Ayesha Asloob Qureshi }

\address{Faculty of Mathematics and Computer Science, Ovidius University, Bd.\ Mamaia 124,
 900527 Constanta, Romania} \email{vivian@univ-ovidius.ro}

\address{Abdus Salam School of Mathematical Sciences, GC University,
Lahore. 68-B, New Muslim Town, Lahore 54600, Pakistan} \email{ayesqi@gmail.com}

\begin{abstract}
With a simple graph $G$ on $[n]$, we associate a binomial ideal $P_G$ generated by diagonal minors of an $n \times n$  matrix $X=(x_{ij})$ of variables. We show that for any graph $G$, $P_G$ is a  prime complete intersection ideal and determine the divisor class group of $K[X]/ P_G$. By using these ideals, one may find a normal domain with free divisor class group of any given rank.
\end{abstract}

\thanks{The first author was supported by the grant UEFISCDI,  PN-II-ID-PCE- 2011-3-1023.}
\maketitle

\section*{Introduction}

Classically, with a simple graph $G$ on the vertex set $[n]$, one associates the so-called edge ideal $I(G)$ in the polynomial ring $K[x_1, \ldots, x_n]$ over a field $K$. Recently, binomial edge ideals have been considered in \cite{HHH} and, independently, in \cite{O}. The binomial edge ideal $J_G$ of $G$ is generated by the binomials $ x_i y_j - x_j y_i \in K[x_1, \ldots, x_n, y_1,\ldots,y_n]$, where $\{i,j\}$ is an edge of $G$. For instance, the ideal of all $2$-minors of a $2 \times n$ matrix of variables is a special example of a binomial edge ideal.

The study of the ideals of $2$-minors of  matrices of variables is motivated by their relevance in algebraic statistics and other fields as it was shown in \cite{DES}. For more recent results on ideals generated by $2$-minors one may consult \cite{S}, \cite{HH}, \cite{AAQ}.

In this paper, we introduce a new class of  ideals of $2$-minors associated with graphs.
Let $X=(x_{ij})$ be an  $n\times n$-matrix of variables and $S=K[X]$ the polynomial ring over a field $K$ in the variables $\{x_{ij}\}_{1\leq i,j\leq n}.$ Let $G$ be a simple graph on the vertex set $[n].$ With this graph we associate an ideal generated by  diagonal  $2$-minors of $X$ in the following way. For   $1\leq i<j\leq n$  we denote by  $f_{ij}$ the diagonal $2$-minor of $X$ given by the elements at the intersections of the rows $i,j$ and the columns $i,j$, that is, $f_{ij}=x_{ii}x_{jj}-x_{ij}x_{ji}.$ Let $P_G$ be the ideal of $S$ generated by the binomials $f_{ij}$ where $\{i,j\}$ is an edge of $G$.

With respect to the lexicographical order on $S$ induced by the natural order of variables, namely $x_{11} > x_{12} > \cdots > x_{1n}>x_{21}> \cdots > x_{nn}$, the reduced Gr\"obner basis consists of binomials of degree at most 4 which have squarefree initial monomials, as we show in Theorem~\ref{degree4}. But if we consider the reverse lexicographical order induced by the natural order of the variables,  then the generators of $P_G$ form a Gr\"obner basis of $P_G$ and moreover, they form even a regular sequence, therefore $\height (P_G)$ equals  the number of edges of $G$ and $\ini_\prec(P_G)$ is squarefree. Here $\prec$ denotes the reverse lexicographic order.

We show in Proposition~\ref{prime} that $P_G$ is a prime ideal, hence the ring $R_G = S/P_G$ is a normal domain.

In the last section we study the divisor class group $\Cl(R_G)$. We show in Theorem~\ref{classgroup} that $\Cl(R_G)$ is free and we express its rank in terms of the graph's data. Finally, in Proposition~\ref{bound}, we give sharp bounds for the possible rank of $\Cl(R_G)$ when $G$ has a given number of edges. Every abelian group is the class group of a Krull domain as shown by Claborn \cite{Cla}.  By using ideals generated by diagonal $2$-minors, one may find an example of a normal domain with free divisor class group of any given rank.

\section{Gr\"obner bases of ideals generated by diagonal $2$-minors} \label{GbasisDiag}
Let $X=(x_{ij})_{1\leq i,j\leq n}$ be a square matrix of variables and $S=K[X]$ the polynomial ring over a field $K.$ For any $1\leq i< j\leq n,$ let $f_{ij}=x_{ii}x_{jj}-x_{ij}x_{ji}.$
Let $G$ be a  simple graph on the vertex set $[n]$ and  $P_G=(f_{ij}: \{i,j\}\in E(G))$ where $E(G)$ is the set of edges of $G.$ In this section we  compute the Gr\"obner bases of  $P_G$ with respect to the lexicographic  and the reverse lexicographic order on $S$.

Usually, when we study ideals generated by minors of  matrices of variables, one uses the lexicographic order induced by the natural order of variables, namely, row by row from left to right. This monomial order selects as initial monomial of each $2$-minor of $X$ the product of variables of the main diagonal. It will turn out that this order gives a rather large Gr\"obner basis for $P_G$. But if we consider a monomial order on $S$ which selects the product of the variables on the anti-diagonal of each $2$-minor as initial monomial, then the generators of $P_G$ form a Gr\"obner basis. More precisely, let us consider the reverse lexicographic order on the ring $S$.
 With respect to this order, $\ini_{\prec}f_{ij}=x_{ij}x_{ji}$, that is, the initial monomial comes from the anti-diagonal of the minor $f_{ij}$ for any $1\leq i< j\leq n$.

\begin{Proposition}
Let $G$ be a simple graph on the vertex set $[n]$ with the edge set $E(G)$ and let $P_G=(f_{ij}: \{i,j\}\in E(G))$ be the binomial ideal generated by the diagonal $2$-minors
associated with $G$. Then the set of generators of $P_G$ is the reduced Gr\"obner basis of $P_G$ with respect to the reverse lexicographic order.
Moreover, $P_G$ is a complete intersection of $\height(P_G)=|E(G)|.$
\end{Proposition}

\begin{proof}
All claims follow immediately if we notice that the initial monomials of $f_{ij}$ with respect to $\prec$ form a regular sequence.
\end{proof}

We now consider the lexicographic order induced by the natural order of variables, namely,
\[
x_{11} > x_{1 2} > \cdots >  x_{1n} > x_{21} > x_{2 2}> \cdots > x_{2n} > \cdots > x_{n1}> x_{n 2}> \cdots > x_{nn}.
\]
 With respect to this order, $\ini_{<}f_{ij}=x_{ii}x_{jj},$ in other words, the initial monomial comes from the main diagonal of the minor $f_{ij}$.

\begin{Theorem}
\label{degree4}
Let $G$ be a simple graph on the set $[n]$ and let $P_G$ be its associated ideal. The initial ideal of $P_G$ with respect to the lexicographic order induced by the natural order of indeterminates is generated by squarefree monomials of degree at most 4.
\end{Theorem}

\begin{proof}
We compute a Gr\"obner basis $\mathcal{G}_{\lex}$ of $P_G$ by applying Buchberger's criterion.

We first compute the $S$-polynomials of the generators of $P_G$. Let $f_{ij}$ and $f_{kl}$ be two binomials in the generating set $P_G$. We consider the non trivial case when $\gcd(\ini_{<}(f_{ij}),\ini_{<}(f_{kl})) \neq 1$. We may have one of the following possibilities:
\begin{enumerate}
\item[(i)] $x_{jj}=x_{kk}$ (or $x_{ii} = x_{ll}$),
\item[(ii)] $x_{ii}=x_{kk}$,
\item[(iii)] $x_{jj}=x_{ll}$.
\end{enumerate}

Consider the case (i) as shown in Figure~\ref{case1}. Then, the $S$-polynomial $S(f_{ij}, f_{jl}) = x_{ii} x_{jl} x_{lj} - x_{ll} x_{ij} x_{ji}$ must be added to $\mathcal{G}_{\lex}$, since none of its monomials is divisible by a monomial of the form $x_{aa} x_{bb}$.

\begin{figure}[hbt]
\begin{center}
\psset{unit=0.6cm}
\begin{pspicture}(4.5,-1.5)(4.5,3)
\rput(-8.5,0)
{
\pspolygon(5,1)(5,2.75)(7,2.75)(7,1)
\pspolygon(7,-0.75)(7,1)(9,1)(9,-0.75)
\psline[linestyle=dashed](5,2.75)(7,-0.75)
\psline[linestyle=dashed](5,2.75)(9,1)
\psline[linestyle=dashed](9,1)(7,-0.75)
\rput(4.5,2.75){$x_{ii}$}
\rput(7.35,2.75){$x_{ij}$}
\rput(4.5,1){$x_{ji}$}
\rput(7.35,1.3){$x_{jj}$}
\rput(9.35,1){$x_{jl}$}
\rput(9.35,-1){$x_{ll}$}
\rput(6.5,-1){$x_{lj}$}
\rput(7,-2){case(i)}
}
\rput(-2.5,0)
{
\pspolygon(5,1)(5,2.75)(7,2.75)(7,1)
\pspolygon(5,-0.75)(5,2.75,1)(9,2.75)(9,-0.75)
\psline[linestyle=dashed](7,2.75)(5,1)
\psline[linestyle=dashed](7,2.75)(9,-0.75)
\psline[linestyle=dashed](5,1)(9,-0.75)
\rput(4.5,3){$x_{ii}$}
\rput(7.35,3){$x_{ij}$}
\rput(4.5,1){$x_{ji}$}
\rput(7.4,1){$x_{jj}$}
\rput(9.35,3){$x_{il}$}
\rput(9.35,-1){$x_{ll}$}
\rput(4.5,-1){$x_{li}$}
\rput(7,-2){case(ii)}
}
\rput(4.5,0)
{
\pspolygon(7,-0.75)(9,-0.75)(9,1)(7,1)
\pspolygon(5,-0.75)(5,2.75,1)(9,2.75)(9,-0.75)
\psline[linestyle=dashed](5,2.75)(7,-0.75)
\psline[linestyle=dashed](5,2.75)(9,1)
\psline[linestyle=dashed](9,1)(7,-0.75)
\rput(4.5,2.9){$x_{ii}$}
\rput(9.35,2.9){$x_{ij}$}
\rput(7,1.3){$x_{kk}$}
\rput(7,-1){$x_{jk}$}
\rput(9.4,1){$x_{kj}$}
\rput(9.4,-1){$x_{jj}$}
\rput(4.5,-1){$x_{ji}$}
\rput(7,-2){case(iii)}
}
\end{pspicture}
\end{center}
\caption{}\label{case1}
\end{figure}

For case (ii), see Figure~\ref{case1}, we may assume that, for instance, $j<l$, and we get $S(f_{ij}, f_{il}) = x_{jj} x_{il} x_{li} -  x_{ll} x_{ij} x_{ji}$ that should also be added to $\mathcal{G}_{\lex}$. Moreover, $\ini_< (S(f_{ij}, f_{il})) = x_{ll} x_{ij} x_{ji}$.

In case (iii), we may choose $i<k$, and get $S(f_{ij}, f_{kj}) = x_{ii} x_{jk} x_{kj} - x_{kk} x_{ji} x_{ij}$ in  $\mathcal{G}_{\lex}$, with $\ini_< (S(f_{ij}, f_{kj})) = x_{ii} x_{jk} x_{kj}$; see Figure~\ref{case1}.

By the above computation of $S$-polynomials, we have got all the binomials of degree 3 which belong to the Gr\"obner basis of $P_G$. Next we investigate the $S$-polynomials $S(g,f_{ij})$, where $g$ is a binomial of degree 3 and $f_{ij}$ is a quadratic binomial. We are going to discuss only those cases when the $S$-polynomial does not reduce to 0, and hence contributes to the Gr\"obner basis.

Case 1: Let $g= x_{ii} x_{jl} x_{lj} - x_{ll} x_{ij} x_{ji}$ with $i<j<l$ and $f_{iq}=x_{ii}x_{qq} - x_{qi}x_{iq}$ with $i<q$. We have $S(f_{iq},g) = x_{jl} x_{lj} x_{qi} x_{iq}- x_{qq} x_{ll} x_{ij} x_{ji}$. If $\ini_< (S(g, f_{iq}))$ is the first monomial, that is, $q<j$, then we add $S(g, f_{iq})$ to $\mathcal{G}_{\lex}$ since $x_{jl} x_{lj} x_{qi} x_{iq}$ is not divisible by any of the previous initial monomials which we have obtained so far; see Figure~\ref{case11}.

\begin{figure}[hbt]
\begin{center}
\psset{unit=1cm}
\begin{pspicture}(4.5,0)(4.5,3)
\rput(-2,0)
{
\pspolygon(5,1)(5,3)(7,3)(7,1)
\pspolygon(7,0)(7,1)(8,1)(8,0)
\pspolygon(5,3)(5,2)(6,2)(6,3)
\psline[linestyle=dashed](6,3)(5,2)
\psline[linestyle=dashed](7,0)(8,1)
\psline[linestyle=dashed](6,3)(8,1)
\psline[linestyle=dashed](5,2)(7,0)
\rput(4.6,3.1){$x_{ii}$}
\rput(7.2,3.1){$x_{ij}$}
\rput(4.6,1){$x_{ji}$}
\rput(7.2,1.15){$x_{jj}$}
\rput(8.2,1){$x_{jl}$}
\rput(8.2,-0.2){$x_{ll}$}
\rput(6.8,-0.2){$x_{lj}$}
\rput(6,3.15){$x_{iq}$}
\rput(4.6,2){$x_{qi}$}
\rput(6,1.85){$x_{qq}$}
}
\end{pspicture}
\end{center}
\caption{}\label{case11}
\end{figure}

Otherwise, that is, for $q>j$, we reduce $S(g, f_{iq})$ modulo a binomial of degree 3, namely the one which appears when we take $S(f_{ij}, f_{iq})$, and get
\[
S(g,f_{iq}) = -x_{ll} S(f_{ij}, f_{iq}) - x_{iq} x_{qi} f_{jl}.
\]

Therefore, $S(g,f_{iq})$ reduces to zero. Now it remains to consider $S(g,f_{qi})$, where $q<i$. By proceeding as before, it follows that $S(g,f_{qi})$ reduces to zero modulo a binomial of degree 3 and a binomial of degree 2.

Case 2: Let $g = x_{ll} x_{ij} x_{ji} - x_{jj} x_{il} x_{li}$ and $f_{lq} = x_{ll} x_{qq}-x_{ql}x_{lq}$, with $l<q$. Then we have $ S(g, f_{lq})= x_{ij} x_{ji} x_{lq} x_{ql} - x_{qq} x_{jj} x_{il} x_{li}$. Since  $\ini_< (S(g, f_{lq}))$ is not divisible by any initial monomial obtained so far, we add this polynomial to our Gr\"obner basis. We get the same conclusion if we take $S(g,f_{ql})$ with $l>q$; see Figure~\ref{case22}.

\begin{figure}[hbt]
\begin{center}
\psset{unit=1cm}
\begin{pspicture}(4.5,-0.5)(4.5,3)
\rput(-4.5,0)
{
\pspolygon(5,1)(5,3)(7,3)(7,1)
\pspolygon(7,0)(7,1)(8,1)(8,0)
\pspolygon(5,3)(5,2)(6,2)(6,3)
\psline[linestyle=dashed](6,3)(5,2)
\psline[linestyle=dashed](7,0)(8,1)
\psline[linestyle=dashed](6,3)(8,1)
\psline[linestyle=dashed](5,2)(7,0)
\rput(4.6,3.1){$x_{ii}$}
\rput(7.2,3.1){$x_{il}$}
\rput(4.6,1){$x_{li}$}
\rput(7.2,1.15){$x_{ll}$}
\rput(8.2,1){$x_{lq}$}
\rput(8.2,-0.2){$x_{qq}$}
\rput(6.8,-0.2){$x_{ql}$}
\rput(6,3.15){$x_{ij}$}
\rput(4.6,2){$x_{ji}$}
\rput(6,1.85){$x_{jj}$}
\rput(6.5,-0.75){$l<q$}
}
\rput(0.5,0)
{
\pspolygon(5,0)(5,3)(8,3)(8,0)
\pspolygon(7,0)(7,1)(8,1)(8,0)
\pspolygon(5,3)(5,2)(6,2)(6,3)
\psline[linestyle=dashed](6,3)(5,2)
\psline[linestyle=dashed](7,0)(8,1)
\psline[linestyle=dashed](6,3)(8,1)
\psline[linestyle=dashed](5,2)(7,0)
\rput(4.6,3.1){$x_{ii}$}
\rput(8.2,3.1){$x_{il}$}
\rput(4.6,-0.2){$x_{li}$}
\rput(6.95,1.15){$x_{qq}$}
\rput(8.2,1){$x_{ql}$}
\rput(8.2,-0.2){$x_{ll}$}
\rput(6.85,-0.2){$x_{lq}$}
\rput(6,3.15){$x_{ij}$}
\rput(4.6,1.9){$x_{ji}$}
\rput(6,1.85){$x_{jj}$}
\rput(6.5,-0.75){$q<l$}
}
\end{pspicture}
\end{center}
\caption{}\label{case22}
\end{figure}

Case 3: Let $g = x_{ii} x_{kj} x_{jk} - x_{ij} x_{ji} x_{kk}$ and $f_{iq} = x_{ii} x_{qq} - x_{iq}x_{qi}$ with $i<q$. Then, we have $ S(g, f_{iq})= x_{iq} x_{qi} x_{jk} x_{kj}- x_{ij} x_{ji} x_{kk} x_{qq}$. If $q<j$, then $\ini_< (S(g, f_{iq}))$ is a new initial monomial, hence we add $S(g, f_{iq})$ to our Gr\"obner basis; see Figure~\ref{case33}. Otherwise, that is, if $q>j$, then, as we did in Case 1, we observe that $S(g, f_{iq})$ reduces to zero modulo a binomial of degree 3 and a binomial of degree 2.

\begin{figure}[hbt]
\begin{center}
\psset{unit=1cm}
\begin{pspicture}(4.5,0)(4.5,3)
\rput(-2,0)
{
\pspolygon(5,0)(5,3)(8,3)(8,0)
\pspolygon(7,0)(7,1)(8,1)(8,0)
\pspolygon(5,3)(5,2)(6,2)(6,3)
\psline[linestyle=dashed](6,3)(5,2)
\psline[linestyle=dashed](7,0)(8,1)
\psline[linestyle=dashed](6,3)(8,1)
\psline[linestyle=dashed](5,2)(7,0)
\rput(4.6,3.1){$x_{ii}$}
\rput(8.3,3.1){$x_{ij}$}
\rput(4.6,-0.2){$x_{ji}$}
\rput(6.95,1.15){$x_{kk}$}
\rput(8.3,1){$x_{kj}$}
\rput(8.3,-0.2){$x_{jj}$}
\rput(6.85,-0.2){$x_{jk}$}
\rput(6,3.15){$x_{iq}$}
\rput(4.6,1.9){$x_{qi}$}
\rput(6,1.85){$x_{qq}$}
}
\end{pspicture}
\end{center}
\caption{}\label{case33}
\end{figure}

If $q<i$, we may again reduce $S(g,f_{qi})$. So far, apart from the original generators of $P_G$, we have in the Gr\"obner basis with respect to the  lexicographic order, binomials of degree 3 and 4.

Now we discuss the $S$-polynomial of degree 3 binomials. Let $g= x_{ii} x_{jl} x_{lj} - x_{ll} x_{ij} x_{ji}$ and $h=x_{pp} x_{rq} x_{qr} - x_{rr} x_{pq} x_{qp}$ be any two binomial of the form that we obtained by computing $S$-polynomials of the generators of $P_G$. If $\gcd(\ini_{<}(g),\ini_{<}(h)) \neq 1$, then we either have $x_{ii} = x_{pp}$ or $\{ x_{jl}, x_{lj} \} = \{ x_{qr}, x_{rq} \}$. In the first case, we obtain a binomial of degree 5 which reduces to 0, and in the second case, $S(f,g)$ is a binomial of degree 4 which is also reducible. To understand this, consider the following example. Let $f=x_{qp} x_{pq} x_{rr} - x_{qq} x_{rp} x_{pr}$ and $g = x_{ii} x_{jl} x_{lj} - x_{ll} x_{ij} x_{ji}$ where $f$ and $g$ are obtained as in case (i) and case(ii) of Figure~\ref{case1}, respectively.

\begin{figure}[hbt]
\begin{center}
\psset{unit=1cm}
\begin{pspicture}(4.5,-1)(4.5,3)
\rput(-2.5,0)
{
\pspolygon(5,1)(5,3)(7,3)(7,1)
\pspolygon(7,0)(7,1)(8,1)(8,0)
\pspolygon(8,-1)(8,0)(9,0)(9,-1)
\pspolygon(5,3)(5,2)(6,2)(6,3)
\rput(4.6,3.1){$x_{pp}$}
\rput(7.2,3.1){$x_{pr}$}
\rput(4.6,1.15){$x_{rp}$}
\rput(7.2,1.15){$x_{rr}$}
\rput(8.2,1.15){$x_{rj}$}
\rput(8.2,0.18){$x_{jj}$}
\rput(6.6,0.18){$x_{jr}$}
\rput(6,3.15){$x_{pq}$}
\rput(4.6,2){$x_{qp}$}
\rput(6,1.85){$x_{qq}$}
\rput(9.2,0.18){$x_{jl}$}
\rput(7.8,-1.1){$x_{lj}$}
\rput(9.2,-1.1){$x_{ll}$}
%\rput(6.5,-0.75){$l<q$}
}
\end{pspicture}
\end{center}
\caption{}\label{example}
\end{figure}

If we let $x_{ii} = x_{rr}$, then $S(f,g)$ reduces to 0 with respect to $x_{qp} x_{pq} x_{jr} x_{rj} - x_{qq} x_{jj} x_{rp} x_{pr}$ and $x_{ll} x_{jj} - x_{jl} x_{lj}$; see Figure~\ref{example}. If $ x_{jl} = x_{pq}$ then $S(f,g)$ is the product of $x_{qq}$ and $x_{ii} x_{rp} x_{pr} - x_{rr} x_{ip} x_{pi}$.

By a careful computation of $S$-polynomials in the cases when we consider $S$-polynomials of degree 4 binomials with binomials of degree 2, 3 and 4, we see that they reduce to 0 and hence $\mathcal{G}_{\lex}$ consists of binomials with squarefree initial term of degree at most 4.
\end{proof}

\begin{Proposition} \label{prime}
$P_G$ is a prime ideal, thus $S/P_G$ is a domain.
\end{Proposition}

\begin{proof}
We may assume that $G$ has no isolated vertices. Let $\{1,i\}$ be an edge of $G$ and let $G'$ be the subgraph of $G$ obtained by removing this edge  from $G$. We first claim that $x_{i1}$ is regular on $S/P_G$. Indeed, we have $(P_G, x_{i1}) = (P_{G'}, x_{11} x_{ii}, x_{i1})$. The height of the ideal $(P_{G'}, x_{11} x_{ii}, x_{i1})$ may be obtained by computing the height of its initial ideal with respect to $\prec$, which is $|E(G')|+2 = |E(G)| +1 = \height (P_G) +1$. Consequently, $\height (P_G, x_{i1}) =\height (P_G) +1$, which shows that $x_{i1}$ is regular on $S/ P_G$ since
$P_G$ is a complete intersection.

Now, the claim of the proposition follows if we show that $(S/P_G)_{x_{i1}}$ is a domain. We have $(S/P_G)_{x_{i1}} \iso (S'/P_{G'})[x_{i1}^{-1}]$, where $S'$ is the polynomial ring in the variables $\{x_{ij} : 1 \leq i,j \leq n\} \setminus \{x_{1i}, x_{i1}\}$. Therefore, the proof is finished by applying induction on the number of edges of $G.$
\end{proof}

\begin{Corollary}
Let $G$ be a simple graph on $[n]$. Then the ring $R_G=S/P_G$ is a normal domain.
\end{Corollary}

\begin{proof}
Since $\ini_{\prec} (P_G)$ is a squarefree monomial ideal, the normality follows by applying a well known criterion of Sturmfels \cite[Chapter 13] {St}.
\end{proof}

\section{The divisor class group} \label{divisor}
Let $G$ be a simple graph on $[n]$ and $R_G=S/P_G.$
In the sequel, we are going to determine the divisor class group of $R_G$. We proceed as in the case of classical determinantal rings, see \cite{BV}, \cite{Co}. Another useful reference on computing  class groups of toric varieties is \cite[Chapter 4]{Cox}.

We first choose an element $y\in R_G$ such that $(R_G)_y$ is a factorial ring. Then, by Nagata's Theorem \cite[Corollary 7.2]{Nag}, we deduce that the divisor class group $\Cl (R_G)$ is generated by the classes of the minimal prime ideals of $y$.

For a vertex $i$ of $G$, we denote by $G \setminus \{i\}$ the subgraph of $G$ obtained by removing the vertex $i$ together with all the edges which are incident to $i$. For the next lemma we need some notation. For $i\in V(G)$ we denote by $N(i)$ the set of all the neighbors of $i$, that
is, $N(i)=\{a\in V(G) : \{a,i\}\in E(G)\}$, and for each $a\in N(i),$ we set $E_a^i=\{x_{ai},x_{ia}\}.$

\begin{Lemma}\label{minprimes}
Let $\{i,j\}$ with $i<j$ be an edge of $G$. Then:
\begin{enumerate}
\item[{\em(a)}] $(P_G , x_{ji})$ is an unmixed radical ideal with $\height (P_G, x_{ji}) = \height P_G +1$.
\item[{\em(b)}] The set of the minimal primes of $(P_G, x_{ji})$ is $\mathcal{C}_1 (i,j) \cup \mathcal{C}_2 (i,j)$, where
\[
\mathcal{C}_1 (i,j) = \{ (P_{G \setminus \{i\}}, x_{ii}, x_{ji}, T)\} \]
 where $T \text{ is any  set of variables with }T\subset\bigcup\limits_{a\in N(i)\setminus\{j\}}E_a^i \text{ and }|T\cap E_a^i|=1 $ for all $ a\in N(i)\setminus\{j\}$, and
\[
\mathcal{C}_2 (i,j) = \{ (P_{G \setminus \{j\}}, x_{jj}, x_{ji}, U)\}
\]
where $U \text{ is any  set of variables with }U\subset\bigcup\limits_{b\in N(j)\setminus\{i\}}E_b^j \text{ and }|U\cap E_b^j|=1$  for all $ b\in N(j)\setminus\{i\}.$
\end{enumerate}
\end{Lemma}

\begin{proof}
We have $(P_G, x_{ji}) = (P_{G'}, x_{ii}x_{jj}, x_{ji})$, where $G'$ is the subgraph of $G$ obtained by removing the edge $\{i,j\}$.  Since $\ini_{\prec} (P_{G'}, x_{ii}x_{jj}, x_{ji})$ is generated by a regular sequence of squarefree monomials of length $|E(G')| +2 = |E(G)| +1$, we get (a).

(b) Obviously, by a height argument, the ideals of the two classes are minimal primes of $(P_G, x_{ji})$. Indeed, for example, for $i\in V(G)$ and any set $T$ which defines an ideal of the set $\mathcal{C}_1 (i,j),$ we have
 $\height (P_{G \setminus \{i\}}, x_{ii}, x_{ji}, T)=\height(P_{G\setminus\{i\}})+|N(i)|+1=\height P_G +1. $

   Let $Q$ be a minimal prime of $(P_G, x_{ji})$.  As $x_{ji} \in Q$, we also have $x_{ii} x_{jj} \in Q$, hence $x_{ii} \in Q$ or $x_{jj} \in Q$. Let, for instance, $x_{ii} \in Q$. Then
\[
Q\supset (P_G,x_{ii},x_{ji})=(P_{G\setminus\{i\}},x_{ii},x_{ji},\{x_{ia}x_{ai}\ |\ a\in N(i)\setminus\{j\}\}).
\]
 Now,  one easily sees that $Q$ contains one of the ideals of the class $\mathcal{C}_1 (i,j)$, and since $\height Q =\height P_G +1$, $Q$ must be equal to one of the ideals of the set $\mathcal{C}_1 (i,j)$.
\end{proof}

Let $y = \prod\limits_{\{i,j\} \in E(G) \atop i<j} \overline{x}_{ji}\in S/ P_G$, and let $\overline{Q} = Q/ P_G \subset S/ P_G$ be a minimal prime of $y$. Then $Q$ is a minimal prime of $(P_G, x_{ji})$ for some  $\{ i,j\} \in E(G), i<j$. Thus $Q$ belongs either to $\mathcal{C}_1 (i,j)$ or to $\mathcal{C}_2 (i,j)$. Therefore, the set $\Min(y)$ of the minimal primes of $(y)$ consists of all ideals $\overline{Q} \subset S/P_G$, where $Q \in  \bigcup_{\{i,j\}\in E(G)} \mathcal{C}_1 (i,j) \cup \mathcal{C}_2 (i,j)$.

In order to determine the cardinality of the set $\bigcup_{\{i,j\}\in E(G)} \mathcal{C}_1 (i,j) \cup \mathcal{C}_2 (i,j),$ we observe that it is enough to count how many prime ideals $Q$ contain $x_{ii}$ for each $i\in V(G).$ But this is easy, since such an ideal $Q$ is determined by a set $T\subset \bigcup_{a\in N(i)}E_a^i$ with $|T\cap E_a^i|=1$ for all $a\in N(i)$ and with the property that for at least one variable  $x_{cd}\in T$ we have $c>d.$
 Therefore, there are $2^{\deg i} -1$ minimal primes of $y$ which contain $\overline{x}_{ii}$. Consequently, we have
\[
|\Min (y)| = \sum_{i \in V(G)} (2^{\deg i} -1) = \sum_{i \in V(G)} 2^{\deg i} - n.
\]

\begin{Theorem}\label{classgroup}
The class group $\Cl(R_G)$ is free of rank $ \sum_{i \in V(G)} 2^{\deg i} - n - |E(G)|$.
\end{Theorem}

\begin{proof}
We first notice that in $(R_G)_y$ we have ${\bar x}_{ij} = {\bar x}_{ii} {\bar x}_{jj} {\bar x}_{ji}^{-1}$ for any edge $\{i,j\}$ of $G$ with $i<j$. Therefore,
\[(R_G)_y \iso K[\{x_{ij}: i,j\in [n]\}\setminus\{x_{ij}: \{i,j\}\in E(G), i<j\}]_z,\] where $z={\prod\limits_{\{i,j\}\in E(G)\atop i<j}x_{ji}}$,
which shows that $(R_G)_y$ is a factorial ring. By Nagata's Theorem, it follows that the class group $\Cl (R_G)$ is generated by the classes of the
minimal primes of $y$. By Lemma~\ref{minprimes}, we get the following relations in $\Cl (R_G)$:

\begin{eqnarray}\label{Cleq}
\sum_{p \in \Min(\overline{x}_{ji})} \cl(p) = 0.
\end{eqnarray}

We show that all the relations between the classes of the minimal primes of $y$ are linear combinations of the relations~(\ref{Cleq}). Indeed, suppose
that $\sum_{q \in \Min(y)} m_{q} \cl(q) = 0$ for some integers $m_q.$

This implies that $\sum_{q \in \Min(y)} m_{q} \div(q) = \div(g)$, where $\div(g)$ is a principal divisor in $R_G$. Since, $\div(q)$ are in the kernel of the homomorphism $\Div(R_G) \rightarrow \Div((R_G)_y)$, it follows that $g$ is a unit of $(R_G)_y$, hence $g= \lambda {\prod_{\{i,j\} \in E(G)} \overline{x}_{ji}^{n_{ji}}}$, for some integers $n_{ji}$ and $\lambda \in K \setminus \{0\}$. Hence, we get $\sum_{q \in \Min(y)} m_q \div(q) = \sum_{\{i,j\} \in E(G)} n_{ij} \div(\overline{x}_{ji}) = \sum_{\{i,j\} \in E(G)} n_{ij} (\sum_{p \in \Min(\overline{x}_{ji}) }\div(p))$, thus $\sum_{q \in \Min(y)} m_q \cl(q)$ is a combination of relations of type (\ref{Cleq}) with coefficients $n_{ij}$. Then, by using the relations (\ref{Cleq}) for each $\{i,j\} \in E(G)$, we may express one class $\cl(p)$ where $p \in \Min(\overline{x}_{ji})$ as a combination of the others, and we get the statement of the theorem.
\end{proof}

Now, we would like to answer the following question. Given a connected graph $G$ with $m$ edges, which are the bounds for the rank of the group $\Cl(R_G)$? The answer is given in the following

\begin{Proposition} \label{bound}
Let $G$ be a connected graph with $m$ edges and let $R_G = S/ P_G$, where $P_G$ is the binomial  ideal associated with $G$. Then
\[
2m-1 \leq \rank \Cl(R_G) \leq 2^m -1.
\]
Moreover, $\rank \Cl(R_G) = 2m-1$ if and only if $G$ is the line graph and $\rank \Cl(R_G) =2^m-1$ if and only if $G$ is the star graph.
\end{Proposition}

\begin{proof}
 We show, by induction on $m,$ that if $G$ is an arbitrary graph with $m$ edges, then $\rank \Cl (R_G) \geq 2m-1$. For $m=1$, the claim is obvious. Let $m > 1$ and $G$ a graph with $m$ edges. We remove one edge $\{i,j\}$ of $G$ and let $G'$ be the new graph. It is clear that at least one of the vertices $i$ and $j$ has degree $\geq 2$, since $G$ is connected. Then,
\[
\rank\Cl(R_G) = \sum_{a \in V(G)} (2^{\deg a} -1) -m
\]
\[
= [\sum_{{a \in V(G),\ }\atop{a \neq i,j}} (2^{\deg a} -1) + (2^{\deg i -1} -1 ) + (2^{\deg j -1} -1 ) - (m-1)] + 2^{\deg i -1} + 2^{\deg j -1} -1
\]
\[
=\rank\Cl(R_{G'} )+ 2^{\deg i -1} + 2^{\deg j -1} - 1 \geq 2m -3 +2 = 2m-1.
\]
We show  that the only graph with $m$ edges for which $\rank \Cl(R_G) = 2m-1$ is the line graph with $m+1$ vertices, by induction on $m.$ The step $m=1$ is clear. Let now $m>1$ and assume that $\rank \Cl(R_G) = 2m-1$. Then, by the above inequalities we get
\[
2m-1= \rank\Cl(R_{G'} )+ 2^{\deg i -1} + 2^{\deg j -1} - 1 \geq 2m -3 +2 = 2m-1.
\] Therefore, we must have $\rank\Cl(R_{G'} )=2m-3,$ thus $G^\prime$ is a line, by induction, and one of the vertices $i,j$ has degree $1$ and the other one has degree $2,$ which yields the desired conclusion.

For proving the inequality $\rank \Cl(R_G) \leq 2^m -1$, we proceed by induction on $m$. More precisely, we first show that if $G$ has no vertex of degree $m$, then $\rank \Cl(R_G)< 2^m-1.$ As in the first part of the proof, we remove an edge of $G$, let us say $\{i,j\}$. Then
\[
\rank\Cl (R_G) = \rank \Cl(R_{G'}) + 2^{\deg i -1} + 2^{\deg j -1} -1
\]
\[\leq 2^{m-1} -1 + 2^{\deg i -1} + 2^{\deg j -1} -1 < 2^m -1,
\]
since $G$ has no vertex of degree $m$. Indeed, we get the last inequality as follows
\[
2^{m-1} + 2^{\deg i -1} + 2^{\deg j -1} -2 \leq 2^{m-1} + 2\cdot 2^{m-2} -2 = 2^m -2 < 2^m -1.
\]
The upper bound for $\rank \Cl (R_G)$ is clearly reached by the star graph, that is the graph with the edges $\{1,2\}, \{1,3\}, \ldots, \{1,m\}, \{1, m+1\}$, which is the only one which has a vertex of degree $ m$.
\end{proof}

\begin{Remark}{\em
One may easily see that, in general, not every integer between $2m-1$ and $2^m-1$ can be the rank of $\Cl(R_G)$ for some connected graph $G$ with
$m$ edges. For instance, for $m=4$, the possible ranks of $\Cl(R_G)$ are $7,8,9,10$ and $15.$ }
\end{Remark}

\begin{Remark}{\em
Note that if $G$ is the cycle with $m$ edges, then $\rank \Cl(R_G)=2m.$ Therefore, for each positive integer $n,$ one may find a graph $G$
such that $\Cl(R_G)$ is a free group of $\rank \Cl(R_G)=n$. Indeed, if $n=2m-1,$ we may take $G$ to be the line graph with $m$ edges, and if $n=2m$, we may take the cycle with $m$ edges.
}
\end{Remark}

%\newpage
{}

\end{document}